\documentclass[12pt, a4]{amsart}
\usepackage{geometry}
\usepackage{amsxtra,amssymb,amsthm,amsmath,amscd,mathrsfs, epsfig, eufrak}
\usepackage{amscd, amsmath, mathrsfs, amssymb, amsthm, amsxtra, bbding, epsfig, eufrak, graphicx, latexsym, mathrsfs}
\usepackage[all]{xy}
\allowdisplaybreaks
\usepackage{tikz-cd}

\usepackage{tikz}
\usepackage[normalem]{ulem}
\usepackage{soul}
\usepackage{color}

\setstcolor{red}
\newcommand{\blue }[1]{{\color{blue}#1}}
\setstcolor{blue}
\makeatletter
\@namedef{subjclassname@2020}{%
	\textup{2020} Mathematics Subject Classification}
\makeatother
\usepackage{hyperref}
\hypersetup{hypertex=true,colorlinks=true,linkcolor=cyan,anchorcolor=cyan,citecolor=cyan}

\def \i {{\rm i}}

\def \M {{\mathcal M}}

\def \S {{\mathcal S}}

\def \d {\,{\rm d}}
\def\re{{\Re e\,}}

\def\geq{\geqslant}
\def\le{\leqslant}
\def\ge{\geqslant}

\theoremstyle{plain}
\newtheorem{theorem}{Theorem}[section]
\newtheorem{proposition}[theorem]{Proposition}
\newtheorem{lemma}[theorem]{Lemma}

\theoremstyle{remark}
\newtheorem{remark}{Remark}

\numberwithin{equation}{section}
\addtocounter{footnote}{1}

\numberwithin{equation}{section}

\addtocounter{footnote}{1}

\begin{document}
	
	\title[A note on the large values of  $|\zeta^{(\ell)}(1+{\rm i}t)|$]
	{A note on the large values of  $|\zeta^{(\ell)}(1+{\rm i}t)|$}
	\author[Zikang Dong]{Zikang Dong}
	\address{%
		CNRS LAMA 8050\\
		Laboratoire d'analyse et de math\'ematiques appliqu\'ees\\
		Universit\'e Paris-Est Cr\'eteil\\
		61 avenue du G\'en\'eral de Gaulle\\
		94010 Cr\'eteil Cedex\\
		France
	}
	\email{zikangdong@gmail.com}
		\author[Bin Wei]{Bin Wei}
	\address{Center for Applied Mathematics, Tianjin University, Tianjin 300072, P.R. China}
	\email{bwei@tju.edu.cn}

	\date{\today}
	
	
	\begin{abstract}
		We investigate the large values of the derivatives of the Riemann zeta function $\zeta(s)$ on the 1-line. We give a larger lower bound for $\max_{t\in[T,2T]}|\zeta^{(\ell)}(1+\i t)|$, which improves the previous result established by Yang \cite{YDD}.
	\end{abstract}
	
	\maketitle
	
	\section{Introduction}
    The study of the extreme values of the Riemann zeta function has a long history.
    Over the past decades, quite a few of results has been established.
    The values on the critical line $\sigma=1/2$ was first considered by Titchmarsh \cite{Ti86}, who showed that there exists arbitrarily large $t$ such that for any $\alpha<1/2$ we have $|\zeta(1/2+\i t)|\ge\exp\big((\log t)^\alpha\big)$. We refer to \cite{Mon77,Rama77,Ba86,BS17,BS19,DT18} for results in progression.
    For the critical strip $1/2<\re s<1$, it was also Titchmarsh \cite{Ti28} who first showed that for any $\varepsilon>0$ and fixed $\sigma\in(1/2,1)$, there exist arbitrarily large $t$ such that $|\zeta(\sigma+{\rm i}t)|\ge \exp\{(\log t)^{1-\sigma-\varepsilon}\}.$ We refer to \cite{Le72,Mon77,A16,BS18,DW22} on this topic.
    The study of the values on the 1-line can date back to 1925 when Littlewood \cite{Li25} showed that there exists arbitrarily large $t$ for which
    $|\zeta(1+{\rm i}t)|\geq\{1+o(1)\}{\rm e}^{\gamma}\log_2t .$
    Here and throughout, we denote by $\log_j$ the $j$-th iterated logarithm and by $\gamma$ the Euler constant. For further results, we refer to \cite{Le72,GS06,Ai19,DW22}.

    It also draws wide interests on the extreme values of the derivatives of the Riemann zeta function.
    For any fixed $\ell\in {\mathbb N}$, denote
    $$Z^{(\ell)}(T):=\max_{t\in[T,2T]}|\zeta^{(\ell)}(1+\i t)|.$$
Besides other results, Yang \cite{YDD} recently proved that if $T$ is sufficiently large, then uniformly for $\ell\le (\log T)/(\log_2 T)$, we have
	\begin{align}
		Z^{(\ell)}(T)\ge \frac{{\rm e}^{\gamma} \ell^{\ell}}{(\ell+1)^{\ell+1}}\{\log_2T-\log_3T+O(1)\}^{\ell+1}.
		\label{qe511}
	\end{align}
	In this note, we aim to improve the constant $\ell^{\ell}/(\ell+1)^{\ell+1}$ in \eqref{qe511}. We have the following theorem.

\begin{theorem}\label{th511}
For $T\to\infty$ and $\ell\le(\log T)/(\log_2 T)$, we have
		$$Z^{(\ell)}(T)\ge \frac{{\rm e}^{\gamma}}{\ell+1}(\log_2T)^{\ell+1}\{1+o(1)\}.$$
	\end{theorem}
	\begin{remark}
    Recently, the authors \cite{DW22} proved that
    $$\max_{t\in[\sqrt{T},T]}|\zeta(1+\i t)|\ge{\rm e}^\gamma(\log_2T+\log_3T+c),$$
    where $c$ is a computable constant. Granville and Soundararajan \cite{GS06} predicted that it should still be true for $\max_{t\in[T,2T]}|\zeta(1+\i t)|$.
    These results seems stronger than that of Theorem \ref{th511} with $\ell =0$.
    The reason is that after taking derivatives of the Riemann zeta function, we are no longer able to make use of the multiplicativity of its Dirichlet coefficients as previously did. Nevertheless, Theorem \ref{th511} remains a generalization of Littlewood's initial bound (see \cite{Li25}).
    \end{remark}

    Both Yang's proof and ours employ the resonance method used by Bondarenko and Seip \cite{BS18}.
    For a large $x$, we take
    \begin{align}
    {\mathcal P}:=\prod_{p\le x}p^{b-1}\qquad\text{and}
    \qquad{\mathcal M}:=\{n\in{\mathbb N}:\;n\mid {\mathcal P}\}.\label{PM}
    \end{align}
    The key ingredient of the proof is a weighted reciprocal sum in the form
    $$
    \S(x; \ell):=\sum_{m\in\M}\sum_{k\mid m}\frac{(\log k)^{\ell}}{k},
    $$
    where $\ell\ge 0$ is an integer.
    In \cite{YDD}, Yang divided ${\mathcal M}$ as well as ${\mathcal P}$ into two subsets,
    according to whether $p\le x^{\ell/(\ell+1)}$.
    Our choice is to give a finer division.
    Specifically, let $J\ge 1$ be a positive integer.
    For $0\le j\le J$, denote
    \begin{align}
    \M_j:=\Big\{m\in{\mathbb N}:\;m\mid \prod_{p\le x^{j/J}}p^{b-1}\Big\}.
    \label{Mj}
    \end{align}
    Thus we divide the set $\M$ into $J$ subsets:
    $$
    \mathcal M = \bigsqcup_{j=1}^J (\M_j\setminus\M_{j-1}).
    $$
    By this trick, we are able to enlarge the estimate of $\S(x; \ell)$ with a factor $(1+1/\ell)^{\ell}$.
    We summarize it as the following proposition.

    \begin{proposition}\label{prop531}
Under the previous notation, we have
$$
\frac{1}{|\M|}\S(x; \ell)
\ge \frac{{\rm e}^{\gamma}}{\ell+1}
\bigg\{1+O\bigg(\frac{1}{J} + \frac {J\log_2 x}b + \frac{J^2}{\log x}\bigg)\bigg\}(\log x)^{\ell+1},
$$
uniformly for $x\ge 3$, $b\ge 1$, $J\ge 1$, $\ell\ge 0$
where the implied constant is absolute.
    \end{proposition}

\medskip

\section{Proof of Proposition \ref{prop531}}

The following asymptotic formula plays a key role in the proof of Proposition \ref{prop531}.

\begin{lemma}\label{lm523}
We have
$$
\prod_{p\le x}\sum_{\nu=0}^{b-1}\bigg(1-\frac{\nu}{b}\bigg)\frac{1}{p^\nu}
= \bigg\{1+O\bigg(\frac{\log_2x}{b}+\frac{1}{\log x}\bigg)\bigg\}{\rm e}^{\gamma}\log x
$$
uniformly for $x\ge 3$ and $b\ge 1$,
where the implied constants are absolute.
\end{lemma}

\begin{proof}
	See also \cite[Eq. (15)]{YDD} and \cite[page 129]{BS18}.
    For a fixed prime $p$, we have
\begin{align*}		
\sum_{\nu=0}^{b-1}\bigg(1-\frac{\nu}{b}\bigg)\frac{1}{p^\nu}
& = \bigg(\sum_{\nu\ge0}-\sum_{\nu\ge b}\bigg)\bigg(1-\frac{\nu}{b}\bigg)\frac{1}{p^\nu}
\\
& =\bigg(1-\frac{1}{b(p-1)}\bigg)\bigg(1-\frac1p\bigg)^{-1}+O\bigg(\frac{1}{p^b}\bigg).
\end{align*}
Therefore, we can deduce that
	\begin{align*}
		\prod_{p\le x}\sum_{\nu=0}^{b-1}\bigg(1-\frac{\nu}{b}\bigg)\frac{1}{p^\nu}
		=\bigg\{1+O\bigg(\frac1b\sum_{p\le x}\frac{1}{p-1}\bigg)\bigg\}\prod_{p\le x}\bigg(1-\frac1p\bigg)^{-1}.
	\end{align*}
Then the lemma follows by Mertens' formula
$$
\prod_{p\le x}\bigg(1-\frac1p\bigg)^{-1}
= \bigg\{1+O\bigg(\frac{1}{\log x}\bigg)\bigg\}{\rm e}^{\gamma}\log x,
$$
and the fact that $\sum_{p\le x}\frac{1}{p-1}\ll \log_2x$.
\end{proof}

    Now we are prepared to prove Proposition \ref{prop531}.
    By the construction, the set $\M$ is divisor-closed which means $k\mid m, \, m\in\M$ implies $k\in\M$.
	Then by the definition of $\M_j$ in \eqref{Mj}, we have
	\begin{align*}
	\sum_{m\in\M}\sum_{k\mid m}\frac{(\log k)^{\ell}}{k}
    = \sum_{i=1}^J\sum_{k\in\M_j\setminus\M_{j-1}}\frac{(\log k)^{\ell}}{k}\sum_{\substack{m\in\M\\ k\mid m}} 1.
	\end{align*}
	Note that $k\in\M_j\setminus\M_{j-1}$ implies $k\ge x^{(j-1)/J}$. Therefore
\begin{align}
\sum_{m\in\M}\sum_{k\mid m}\frac{(\log k)^{\ell}}{k}
\ge(\log x)^{\ell} \sum_{j=1}^J\bigg(\frac {j-1}{J}\bigg)^{\ell}
\sum_{k\in\M_j\setminus\M_{j-1}}\frac{1}{k}\sum_{\substack{m\in\M\\ k\mid m}} 1.
\label{qe536}
\end{align}

	For each $0\le i\le J$, we rewrite $m=m_1m_2$ where $m_1$ and $m_2$ lie in $\M_i$ and $\M\setminus\M_i$ accordingly. Then we have
\begin{align*}
		\sum_{k\in\M_{i}}\frac{1}{k}\sum_{\substack{m\in\M\\ k\mid m}} 1
=\sum_{m_1\in\M_{i}}\sum_{k|{m_1}}\frac{1}{k}
\sum_{m_2\in\M\setminus\M_i}1.
	\end{align*}
For the sum over $m_1$, we have
$$
\sum_{m_1\in\M_{i}}\sum_{k\mid m_1}\frac{1}{k}
=\prod_{p\in\M_i}\sum_{\substack{m_1|p^{b-1}\\ k\mid m_1}} \frac{1}{k}
=\prod_{p\le x^{i/J}}\sum_{\nu=0}^{b-1} \frac{b-\nu}{p^\nu}.
$$
For the sum over $m_2$, clearly we have
$$
\sum_{m_2\in\M\setminus\M_i}1=\prod_{x^{i/J}<p\le x}b.
$$
Therefore, we deduce that
	\begin{align*}
		\sum_{k\in\M_{i}}\frac{1}{k}\sum_{\substack{m\in\M\\ k\mid m}} 1
={b^{\pi(x)}}\prod_{p\le x^{i/J}}\sum_{\nu=0}^{b-1}\frac{1}{p^\nu}\bigg(1-\frac{\nu}{b}\bigg).
	\end{align*}
	Note that ${b^{\pi(x)}}=|\M|$. By Lemma \ref{lm523}, we have
\begin{align}\label{Mi}
\sum_{k\in\M_{i}} \frac{1}{k} \sum_{\substack{m\in\M\\ k\mid m}} 1
= \frac iJ |\M| \bigg\{1+O\bigg(\frac{\log_2 x}{b} + \frac{J}{\log x}\bigg)\bigg\} {\rm e}^{\gamma}\log x.
\end{align}

In view of \eqref{Mi}, by taking difference of $\M_{j-1}$ and $\M_j$, we obtain
$$
\sum_{k\in\M_j\setminus\M_{j-1}}\frac{1}{k}\sum_{\substack{m\in\M\\ k\mid m}} 1
= \frac {|\M|}J \bigg\{1+O\bigg(\frac{J\log_2 x}{b} + \frac{J^2}{\log x}\bigg)\bigg\}{\rm e}^{\gamma}\log x.
$$
Inserting this into \eqref{qe536}, we have
$$	
\sum_{m\in\M}\sum_{k\mid m}\frac{(\log k)^{\ell}}{k}
\ge \frac{|\M|}J\sum_{j=1}^J \bigg(\frac {j-1}{J}\bigg)^{\ell}
\bigg\{1+O\bigg(\frac{J\log_2 x}{b}+\frac{J^2}{\log x}\bigg)\bigg\}{\rm e}^{\gamma}(\log x)^{\ell+1}.
$$

Now Proposition \ref{prop531} follows supplied that
$$
\frac{1}{J} \sum_{j=1}^J \bigg(\frac {j-1}{J}\bigg)^{\ell}
= \frac{1}{\ell+1}+O\bigg(\frac{1}{J}\bigg).
$$
While this is trivial by the integral inequalities
$$
\frac{1}{J}\sum_{j=1}^J\bigg(\frac {j-1}{J}\bigg)^{\ell}
\le \int_0^1u^{\ell} \d u
\le \frac{1}{J}\sum_{j=1}^J\bigg(\frac {j-1}{J}\bigg)^{\ell}+\frac{1}{J}.
$$

\section{Proof of Theorem \ref{th511}}
We start with the following lemma, which helps approximate the derivatives of the Riemann zeta function by the Dirichlet polynomials.
\begin{lemma}\label{lm521}
For $T\to\infty$, $T\le t\le 2T$ and $\ell\le(\log T)/(\log_2 T)$, we have that
$$
(-1)^{\ell}\zeta^{(\ell)}(1+\i t)
= \sum_{n\le T}\frac{(\log n)^{\ell}}{n^{1+\i t}} + O\big((\log_2 T)^{\ell}\big),
$$
where the implied constant is absolute.
\end{lemma}
\begin{proof}
This is \cite[Lemma 1]{YDD}, where we have taken $\sigma=1$ and $\varepsilon=(\log_2 T)^{-1}$ as Yang did. See also \cite[Theorem 4.11]{Ti86}.
\end{proof}

To employ the resonance method, we choose the same weight function $\phi(\cdot)$ as that used by Soundararajan \cite[page 471]{S08}.
Thus let $\phi(t)$ be a smooth function compactly supported in $[1,2]$,
such that $0\le \phi(t)\le 1$ always and $\phi(t)=1$ for $t\in(5/4,7/4)$.
Then the Fourier transform of $\phi$ satisfies
$\widehat{\phi}(u)\ll_\alpha{|u|^{-\alpha}}$ for any integer $\alpha\ge1$.

For sufficiently large $T$, we set
$$x=\frac{\log T}{3\log_2T}\qquad\text{and}\qquad b=\lfloor\log_2T\rfloor.$$
Furthermore, we take $\mathcal P$ and $\mathcal M$ as \eqref{PM}.
Note that ${\mathcal P}\le\sqrt T$ by the prime number theorem.
Then we define the resonator
$$R(t):=\sum_{m\in\M}m^{\i t}.$$
Denote
$$M_1(R,T):=\int_{\mathbb R}|R(t)|^2\phi\bigg(\frac tT\bigg)\d t,$$
$$M_2(R,T):=\int_{\mathbb R}(-1)^{\ell}\zeta^{(\ell)}(1+\i t)|R(t)|^2\phi\bigg(\frac tT\bigg)\d t.$$
Since $	{\rm supp}(\phi)\subset[1,2]$, we have that
\begin{align}
Z^{(\ell)}(T)\ge\frac{|M_2(R,T)|}{M_1(R,T)}.\label{z^(l)}
\end{align}

For $M_1(R,T)$, we have
\begin{align*}
	M_1(R,T)=\sum_{m,n\in{\mathcal M}}\int_{\mathbb R}\bigg(\frac mn\bigg)^{\i t}\phi\bigg(\frac tT\bigg)\d t
	=T\sum_{m,n\in{\mathcal M}}\widehat\phi(T\log(n/m)).
\end{align*}
When $m\neq n$, the choice of ${\mathcal P}$ guarantees that $|\log(n/m)|\gg 1/\sqrt T$,
and consequently
\begin{align}
	\widehat\phi(T\log(n/m))\ll\frac{1}{T^2}.\label{qe531}
\end{align}
Thus the off-diagonal terms contributes
\begin{align*}
	T\sum_{\substack{m,n\in{\mathcal M}\\ m\neq n}} \widehat\phi(T\log(n/m))\ll\frac{1}{T}|\M|^2.
\end{align*}
Therefore, we derive that
\begin{align}
	M_1(R,T)=T\widehat\phi(0)|\M|+O\bigg(\frac{1}{T}|\M|^2\bigg).\label{qe532}
\end{align}

For $M_2(R,T)$, by Lemma \ref{lm521} we have
\begin{align}
	M_2(R,T)
	& = \int_{\mathbb R}\bigg(\sum_{k\le T}\frac{(\log k)^{\ell}}{k^{1+\i t}}\bigg)|R(t)|^2\phi\bigg(\frac tT\bigg)\d t+O\left((\log_2 T)^{\ell}M_1(R,T)\right)\nonumber\\
	& = T\sum_{k\le T}\frac{(\log k)^{\ell}}{k}\sum_{m,n\in{\mathcal M}}\widehat\phi(T\log(kn/m))+O\left((\log_2 T)^{\ell}M_1(R,T)\right).\label{qe533}
\end{align}
Similar to \eqref{qe531}, for $kn\neq m$ we have
\begin{align*}
	\widehat\phi(T\log(kn/m))\ll\frac{1}{T^2}.
\end{align*}
Consequently, we obtain that
\begin{align*}
	\sum_{k\le T}\frac{(\log k)^{\ell}}{k}\sum_{\substack{m,n\in{\mathcal M}\\ kn\neq m}} \widehat\phi(T\log(kn/m))
\ll\frac{(\log T)^{\ell+1}}{T^2}|\M|^2.
\end{align*}
Inserting this into \eqref{qe533}, we have
\begin{align*}
	M_2(R,T)=T\widehat\phi(0)\sum_{m\in\M}\sum_{k\mid m}\frac{(\log k)^{\ell}}{k}
+O\bigg(\frac{(\log T)^{\ell+1}}{T}|\M|^2\bigg)+O\left((\log_2 T)^{\ell}M_1(R,T)\right).
\end{align*}
Combining this with \eqref{z^(l)} and \eqref{qe532}, we derive that
\begin{align}
	Z^{(\ell)}(T)\ge\frac{1}{|\M|}\sum_{m\in\M}\sum_{k\mid m}\frac{(\log k)^{\ell}}{k}+O\left((\log_2 T)^{\ell}\right).\label{rateM2M1}
\end{align}
Then the theorem follows by taking $J=\lfloor\tfrac12\log_3T\rfloor$ in Proposition \ref{prop531}.
	\hfill
	$\square$

	\vskip 5mm
	
	\textbf{Acknowledgement}.
	The authors would like to thank professor Jie Wu for his suggestion on exploring this subject and Daodao Yang for some useful discussion.  Zikang Dong is supported by the China Scholarship Council (CSC) for his study in France.

\end{document}